\documentclass[11pt]{amsart}
\usepackage[all]{xy}

\usepackage{amsmath}
\usepackage{amsfonts}
\usepackage{amssymb}
\usepackage{mathrsfs}

\usepackage{verbatim}
\usepackage{url}

\DeclareFontEncoding{OT2}{}{} % to enable usage of cyrillic fonts

\numberwithin{equation}{section}

\def\kb{\overline{K}}

\def\bh{{\mathbb H}}

\def\bz{{\mathbb Z}\,}
\def\bq{{\mathbb Q}}
\def\bg{{\mathbb G}}
\def\spec{{\rm{Spec}}\,}

\def\der{\le\rm{der}}
\def\rad{\e{\rm{rad}}}

\def\be{\kern -.1em}
\def\lbe{\kern -.025em}

\def\pic{{\rm{Pic}}}
\def\hom{{\rm{Hom}\e}}
\def\rhom{{\rm{RHom}}\e}
\def\krn{{\rm{Ker}}\e }

\def\cok{{\rm{Coker}}}
\def\ms{\mathscr }
\def\tor{\e\rm{tor}}

\def\Gtil{{\widetilde{G}}}

\def\ra{\rightarrow}
\def\e{\kern 0.08em}
\def\le{\kern 0.03em}
\def\ng{\kern -0.04em}

\def\g{\varGamma}

\def\krn{{\rm{Ker}}\,}
\def\cok{{\rm{Coker}}\,}

\newtheorem{lemma}{Lemma}[section]
\newtheorem{theorem}[lemma]{Theorem}
\newtheorem{corollary}[lemma]{Corollary}
\newtheorem{proposition}[lemma]{Proposition}
\theoremstyle{definition}
\newtheorem{definition}[lemma]{Definition}

\theoremstyle{remark}
\newtheorem{remark}[lemma]{Remark}
\newtheorem{remarks}[lemma]{Remarks}

\newtheorem{examples}[lemma]{Examples}

\begin{document}

\title[Flasque resolutions of reductive group schemes]{Flasque resolutions of
reductive group schemes}


\author{Cristian D. Gonz\'alez-Avil\'es}
\address{Departamento de Matem\'aticas, Universidad de La Serena, Chile}
\email{cgonzalez@userena.cl}


\thanks{The author is partially supported by Fondecyt grant
1080025}

\maketitle

\begin{abstract} We generalize J.-L.Colliot-Th\'el\`ene's
construction of flasque resolutions of reductive group schemes over
a field to a broad class of base schemes.
\end{abstract}

\section{Introduction}

In \cite{ct}, J.-L. Colliot-Th\'el\`ene constructed so-called
flasque resolutions of (connected) reductive algebraic groups over a
field. The case of tori was previously investigated by J.-L.
Colliot-Th\'el\`ene and J.-J. Sansuc in \cite{cts1} and \cite{cts2}.
The object of this paper is to extend Colliot-Th\'el\`ene's
construction to reductive group schemes over a broad class of base
schemes $S$, including all connected and locally noetherian normal
schemes. This is the subject of Section 3. In particular, we show
that the functor which assigns to a reductive $S$-group scheme $G$
its (Borovoi) algebraic fundamental $S$-group scheme $\pi_{1}(G)$ is
{\it exact} (see Theorem 3.14). In Section 4, which concludes the
paper, we give some applications of flasque resolutions to Borovoi's
abelianized cohomology theory over $S$. Further applications (over a
Dedekind base) are given in \cite{ga2}.

\section*{Acknowledgements}
I am very grateful to Brian Conrad for providing the proof of the
key Proposition 2.8 and for helping me with the proofs of a number
of results contained in Section 2. I also thank two anonymous
referees for many valuable suggestions.

\section{Preliminaries}

Recall \cite{ega}, $\text{IV}_{\be 1}$, Chap.0, \S 23.2.1, that a local ring $A$
with residue field $k$ is called {\it geometrically unibranched} if
$A_{\text{red}}$ is integral and the integral closure of
$A_{\text{red}}$ (in its field of fractions) is a local ring whose
residue field is a purely inseparable extension of $k$. An equivalent condition is that the spectrum of the strict henselization of $A$ is an irreducible scheme \cite{ega}, $\text{IV}_{\be 4}$, 18.8.15. A scheme $S$
is called geometrically unibranched if the local ring ${\mathcal
O}_{S,s}$ is geometrically unibranched for every point $s$ of $S$.
Clearly, a local integral and integrally closed ring is
geometrically unibranched, whence a normal scheme is geometrically
unibranched.

We will work over an arbitrary non-empty scheme when possible, but many of our results will apply to the following restricted class of schemes:

\begin{definition} A non-empty scheme $S$ is called {\it admissible} if
it is connected, locally noetherian and geometrically unibranched.
If $S=\spec A$ is affine, we will also say that $A$ is admissible.
\end{definition}

\smallskip

Now let $S$ be a non-empty scheme. All $S$-group schemes appearing
in this paper will be of finite type. If $M$ is an $S$-group scheme
of finite type and of multiplicative type, then $M$ is
quasi-isotrivial, i.e., there exists a surjetive \'etale morphism
$S^{\e\prime}\ra S$ such that $M\be\times_{S} S^{\e\prime}$ is
diagonalisable \cite{sga3}, X, Corollary 4.5. Further, the functor
$\underline{\hom}_{\, S\text{-gr}}(M,\bg_{m,S})$ is represented by a
finitely generated twisted constant $S$-group scheme which is
denoted by $M^{*}$ \cite{sga3}, X, Theorem 5.6. On the other hand,
if $X$ is a finitely generated twisted constant $S$-group scheme,
the functor $\underline{\hom}_{\, S\text{-gr}}(X,\bg_{m,S})$ is
represented by an $S$-group scheme of finite type and of
multiplicative type which is denoted by $X^{*}$ \cite{sga3}, X,
Proposition 5.3. The functors $M\to M^{*}$ and $X\to X^{*}$ are
mutually quasi-inverse anti-equivalences between the categories of
$S$-group schemes of finite type and of multiplicative type and that
of finitely generated twisted constant $S$-group schemes
\cite{sga3}, X, Corollary 5.9. Further, $M\to M^{*}$ and $X\to
X^{*}$ are exact functors (see \cite{sga3}, VIII, Theorem 3.1, and
use faithfully flat descent). Now, if $T$ is an $S$-torus, the
functor $\underline{\hom}_{\, S\text{-gr}}(\bg_{m,S},T\e)$ is
represented by a (finitely generated) twisted constant $S$-group
scheme which is denoted by $T_{*}$ \cite{sga3}, X, Theorem 5.6.
There exist canonical isomorphisms
\begin{equation}\label{bas0}
T^{*}\simeq(T_{*})^{\vee}:=\hom_{S\lbe\text{-gr}}(T_{*},\bz_{\! S})
\end{equation}
and
\begin{equation}\label{bas}
T\simeq\hom_{ S\lbe\text{-gr}}(T^{*},\bg_{m,S})\simeq
T_{*}\otimes_{\e\bz_{\! S}}\!\bg_{m,S}.
\end{equation}
Note that, if $\mu$ is a finite $S$-group scheme of multiplicative
type, then
$$
\mu_{*}:=\hom_{S\lbe\text{-gr}}(\bg_{m,S},\mu)=0.
$$

Now assume that $S$ is admissible. Then an $S$-group scheme of
finite type and of multiplicative type is, in fact, {\it isotrivial}
\cite{sga3}, X, Theorem 5.16, i.e., it is split by a {\it finite}
surjetive \'etale morphism $S^{\e\prime}\ra S$. By \cite{sga3}, X,
beginning of \S 1, $S^{\e\prime}$ can be assumed to be connected and
Galois over $S$ with (finite) Galois group $\varDelta$. The
correspondence $M\mapsto M^{*}\be(S^{\e\prime})$ defines an
anti-equivalence between the categories of $S$-group schemes of
finite type and of multiplicative type split by $S^{\e\prime}\ra S$
and that of finitely generated $\varDelta$-modules \cite{sga3}, X,
Corollary 1.2. The reverse anti-equivalence transforms exact
sequences of finitely generated $\varDelta$-modules into exact
sequences of $S$-groups of finite type and of multiplicative type,
inducing exact sequences of fppf sheaves on $S$. Now, if $\ms
L_{\be\varDelta}$ is the category of $\bz$-free and finitely
generated $\varDelta$-modules, then $\ms L_{\be\varDelta}$ is
anti-equivalent to itself via linear duality $X\mapsto
X^{\vee}:=\hom(X,\bz\!)$. An object $X\in\ms L_{\varDelta}$ is
called a {\it permutation} module if it admits a $\varDelta$-stable
$\bz$-basis. It is called {\it invertible} if it is a
$\varDelta$-direct summand of a permutation module. A module $X$ in $\ms L_{\varDelta}$ is called {\it flasque} if the Tate cohomology group $\widehat{H}^{-1}(\varTheta,X)$ is zero for every subgroup $\varTheta$ of $\varDelta$. Now let $T$ be an
$S$-torus. Then $T$ is called flasque (respectively,
quasi-trivial, invertible) if there exists a connected Galois cover
$S^{\e\prime}\ra S$, with Galois group $\varDelta$, splitting $T$
and such that $T^{*}\be(S^{\e\prime})$ is a flasque (respectively,
permutation, invertible) $\varDelta$-module. An
invertible $S$-torus (in particular, a quasi-trivial $S$-torus) is
flasque. See \cite{cts2}, \S\S 1 and 2, for the
basic theory of flasque $S$-tori.

We often restrict our attention to admissible base schemes because
of the following lemma.

\begin{lemma} Let $S$ be an admissible scheme. Then any extension of a quasi-trivial
$S$-torus by a flasque $S$-torus is split.
\end{lemma}
\begin{proof} Let $R$ (resp., $F$) be a quasi-trivial (resp., flasque) $S$-torus and let $1\ra F\ra E\ra R\ra 1$ be an extension of $R$ by $F$, i.e., $E$ is an $S$-group scheme and $R$ is the quotient of $E$ by $F$ for the fppf topology. Then $E$ is an $S$-torus (see \cite{sga3}, XVII, Proposition 7.1.1
and its proof) and the given sequence induces an extension of \'etale twisted constant $S$-group schemes $1\ra R^{*}\ra E^{*}\ra F^{*}\ra 1$. Since $S$ is connected and geometrically unibranched and $F$ is flasque, the above extension is split \cite{cts2}, Lemma 2.1(i), and hence so is the original extension.
\end{proof}

\begin{remark} The lemma fails to hold if $S$ is not admissible. For example,
if $S$ is an algebraic curve over a field having an ordinary double
point, then ${\rm{Ext}}^{1}_{S}(\bg_{m,S},\bg_{m,S})\simeq
H^{1}_{\rm{\acute{e}t}}(S,\bz)\neq 0$ (see \cite{cts2}, p.160, and
\cite{sga4}, IX, Remark 3.7).
\end{remark}

Let $S$ be a non-empty scheme. An $S$-group scheme $G$ is called {\it reductive} (respectively, {\it semisimple}, {\it simply-connected}) if it is affine and
smooth over $S$ and its geometric fibers are {\it connected} reductive (resp., semisimple, simply-connected) algebraic groups \cite{sga3}, XIX, Definition 2.7. If $G$ is a reductive $S$-group scheme, the functor $\underline{\hom}_{\e S\lbe\text{-gr}}(G,\bg_{m,S})$ is represented by a twisted-constant $S$-group scheme which is denoted by $G^{\le
*}$ \cite{sga3}, XXII, Theorem
6.2.1(i). The derived group of $G$ \cite{sga3}, XXII, Theorem
6.2.1(iv), will be denoted by $G^{\der}$. It is a normal semisimple
subgroup scheme of $G$ and the quotient
$$
G^{\tor}:=G/G^{\der}
$$
is an $S$-torus (in \cite{sga3}, XXII, 6.2, $G^{\tor}$ is denoted by
$\rm{corad}(G\e)$ and called the {\it coradical} of $G$). Note that,
since $G^{\der\le*}=0$, the exact sequence of reductive $S$-group
schemes
\begin{equation}\label{seq1}
1\ra G^{\der}\ra G\ra G^{\tor}\ra 1
\end{equation}
induces an equality $G^{\le *}=G^{\tor\le*}$.

We will write $Z(G)$ for the center of $G$, i.e., the $S$-subgroup
scheme of multiplicative type of $G$ which represents the functor
$\underline{{\rm{Centr}}}_{\e G}$ defined in \cite{sga3},
$\text{VI}_{\text{B}}$, Definition 6.1(iii). See \cite{sga3}, XII,
Proposition 4.11, and/or \cite{con}, \S3.3.

\smallskip

If $G$ is any semisimple $S$-group scheme, there exists a
simply-connected $S$-group scheme $\Gtil$ and a central isogeny
$\varphi\colon\Gtil\ra G$ (see \cite{con}, Exercise
6.5.3(i)\footnote{The solution to this exercise can be found in
\cite{har}, beginning of \S1.2.}). The pair $(\Gtil,\varphi)$ has
the following universal mapping property (which determines it up to
unique isomorphism): if $G^{\e\prime}$ is a semisimple $S$-group
scheme and $p\colon G^{\e\prime}\ra G$ is a central isogeny, then
there exists a unique central isogeny $f\colon \Gtil\ra
G^{\e\prime}$ such that $p\circ f=\varphi$. Further, the formation
of $(\Gtil,\varphi)$ commutes with arbitrary extensions of the base.
The $S$-group scheme $\Gtil$ is called the {\it simply-connected
central cover} of $\,G$, and the finite $S$-group scheme of
multiplicative type $\mu:=\krn\varphi$ is called the {\it
fundamental group} of $G$. Note that $\mu$ can be non-smooth, but
its Cartier dual $\mu^{\lbe *}$ is \'etale over $S\,$\footnote{The
reader should be warned that many authors call $\mu^{\lbe *}$ the
fundamental group of $G$. Perhaps more appropriately, $\mu^{\lbe *}$
should be called the {\it \'etale} fundamental group of $G$.}.

\begin{proposition} Let $S$ be a non-empty scheme. Then any central extension of a
simply-connected $S$-group scheme by an $S$-group scheme of
multiplicative type is split.
\end{proposition}
\begin{proof} Let $\mathcal E\colon 1\ra M\ra E\overset{\pi}\ra G \ra 1$ be a
central extension of $S$-group schemes, where $G$ is
simply-connected and $M$ is of multiplicative type. By \cite{bor2},
Lemma B.2, there exists an $S$-torus $T$ and an embedding
$M\hookrightarrow T$. Let $E_{1}$ be the pushout (see, e.g.,
\cite{wbl}, p.77) of $M\hookrightarrow E$ and $M\hookrightarrow T$.
Then there exists an exact commutative diagram
$$
\xymatrix{  & 1\ar[d]&1\ar[d] & &\\
1\ar[r] &  M  \ar[r] \ar[d] &
E\ar[r]^{\pi}\ar[d] &  G \ar[r]\ar@{=}[d] & 1\\
1\ar[r] & T \ar[r]\ar[d] &  E_{1} \ar[r]^{\pi_{1}}\ar[d] &  G \ar[r] & 1\\
& T_{1} \ar[d]\ar@{=}[r]&T_{1}\ar[d] & &\\
& 1&1 & &,}
$$
where $T_{1}=T/M$ is an $S$-torus. The $S$-group scheme $E_{1}$ is
affine and smooth over $S$ and has connected reductive fibers, i.e.,
is a reductive $S$-group scheme. Now $\pi_{\lbe 1}$ induces a
surjection $\pi_{1}^{\e\prime}\colon E_{1}^{\der}\twoheadrightarrow
G^{\der}=G$, and we have a central extension
$$
1\ra T\cap E_{1}^{\der}\ra E_{1}^{\der} \overset{\!\pi_{\be 1}^{\e\prime}}\ra G \ra 1.
$$
Since $E_{1}^{\der}$ is semisimple, $\pi_{1}^{\e\prime}$ is in fact
a central isogeny. Thus, since $G$ is simply-connected,
$\pi_{1}^{\e\prime}$ is an isomorphism. On the other hand, since
$T_{1}$ is commutative, there exists an embedding
$E_{1}^{\der}\hookrightarrow E$. Now the composition
$$
G\overset{(\be\pi_{1}^{\e\prime}\lbe)^{\lbe -\be 1}}
\longrightarrow E_{1}^{\der}\hookrightarrow E
$$
splits $\mathcal E$.
\end{proof}

\begin{remark} The above proof shows that, if
$1\ra T\ra E\overset{\pi}\ra G \ra 1$ is a central extension of a
simply-connected $S$-group scheme $G$ by an $S$-torus $T$, then $E$
is a reductive $S$-group scheme and $\pi$ induces an isomorphism
$E^{\der}\simeq G$.
\end{remark}

\begin{corollary} Let $S$ be a non-empty scheme and let $\lambda\colon G\ra H$ be
a homomorphism between semisimple $S$-group schemes. Then there
exists a unique homomorphism $\widetilde{\lambda}\colon
\widetilde{G}\ra\widetilde{H}$ such that the diagram
$$
\xymatrix{\widetilde{G}\ar[r]^{\widetilde{\lambda}}
\ar[d]^{\varphi_{_{\be G}}} & \widetilde{H}\ar[d]^{\varphi_{_{\be H}}}\\
G\ar[r]^{\lambda}& H}
$$
commutes.
\end{corollary}
\begin{proof} Uniqueness follows from the fact that there are no nontrivial
homomorphisms from a simply-connected $S$-group scheme to a
commutative $S$-group scheme (more precisely,
$\hom_{S\lbe\text{-gr}}\big(\Gtil,\mu_{\lbe\lbe H}\big)=1$). Now the
pullback of the central extension $1\ra\mu_{\lbe\lbe  H}\ra
\widetilde{H}\ra H\ra 1$ along the composition
$\lambda\circ\varphi_{_{\be G}}\colon \widetilde{G}\ra H$ is a
central extension of $\widetilde{G}$ by $\mu_{\lbe\lbe  H}$. By the
proposition, this extension is split, whence there exists a section
$\widetilde{G}\ra \widetilde{H}\times_{H}\widetilde{G}$ to the
canonical projection $\text{pr}_{2}\colon
\widetilde{H}\times_{H}\widetilde{G}\ra \widetilde{G}$. The
composition $\widetilde{G}\ra
\widetilde{H}\times_{H}\widetilde{G}\overset{\text{pr}_{\lbe
1}}\longrightarrow \widetilde{H}$ is the desired map
$\widetilde{\lambda}$.
\end{proof}

\begin{remark} If $\lambda\colon G_{1}\ra G_{2}$ and $\kappa\colon G_{2}\ra
G_{3}$ are homomorphisms between semisimple $S$-group schemes, then
the uniqueness statement in the corollary implies that
$\widetilde{\kappa\circ\lambda}=\widetilde{\kappa}\circ
\widetilde{\lambda}$.
\end{remark}

\begin{proposition} Let $S$ be a non-empty scheme and let
$$
1\ra G_{1}\ra G _{ 2}\ra  G _{ 3}\ra 1
$$
be an exact sequence of semisimple $S$-group schemes. Then the
induced sequence of simply-connected central covers
$$
1\ra\widetilde{ G }_{ 1}\ra\widetilde{ G }_{ 2}\ra\widetilde{ G }_{
3}\ra 1
$$
is exact.
\end{proposition}
\begin{proof} By Remark 2.7, $1\ra\widetilde{ G }_{ 1}\ra \widetilde{ G }_{
2}\ra\widetilde{ G }_{ 3}\ra 1 $ is a complex and to prove its
exactness we need only check that the induced morphism $\widetilde{
G }_{ 2}/\widetilde{ G }_{ 1}\twoheadrightarrow\widetilde{G}_{ 3}$
(which is an isogeny for dimension reasons) is an isomorphism. This
can be checked on fibers over the base, so we are reduced to the
case $S=\spec k$, where $k$ is an algebraically closed field. Let
$C=G_{2}\times_{G_{3}}\Gtil_{3}$ and consider the exact sequence of
$k$-group schemes
\begin{equation}\label{c1}
1\ra G _{ 1}\ra C\ra\widetilde{G}_{ 3}\ra 1.
\end{equation}
The $k$-group $C$ is smooth and connected, since there exists a
smooth quotient map $C\twoheadrightarrow\Gtil_{3}$ whose kernel
$G_{1}$ is connected. Further, there exists a central extension
$$
1\ra\mu_{3}\ra C\ra G _{2}\ra 1.
$$
It follows that $C$ is connected and semisimple. Thus \eqref{c1} is
an exact sequence of smooth connected semisimple algebraic
$k$-groups. By the general structure theorem of such groups, $G_{1}$
embeds into $C$ as an almost direct factor. More precisely, let $ H
$ denote the smooth connected normal subgroup of $C$ generated by
the minimal nontrivial smooth connected normal subgroups of $C$ not
contained in $G_{1}$. Then $H$ is connected, semisimple and
centralizes $G_{ 1}$, and there exists a central isogeny $ G _{
1}\be\times_{k} H \twoheadrightarrow C$. In particular, $H
\twoheadrightarrow\Gtil_{ 3}$ is an isogeny (for dimension reasons
again) and has kernel $H\cap G _{ 1}$, which is central in $H $.
Therefore $ H \ra\Gtil_{3}$ is an isomorphism. That is, the
projection $C\twoheadrightarrow\widetilde{ G }_{ 3}$ has a section
which identifies $\widetilde{ G }_{3}$ with $ H $ and $ H $
centralizes $ G _{ 1}$, whence $ H \cap\, G _{1}=\lbrace 1\rbrace$.
In other words, \eqref{c1} splits. Thus $\widetilde{ G }_{
2}\simeq\widetilde{C}\simeq\widetilde{ G }_{ 1}\times_{k}\widetilde{
G }_{ 3}$, compatibly with the natural maps (compatibility can be
checked after composing with isogenies down to terms in $1\ra G
_{1}\ra G _{ 2}\ra G _{ 3}\ra 1$, since our groups are smooth and
connected).
\end{proof}

Now, if $G$ is any reductive $S$-group scheme, $\Gtil$ will denote
the simply-connected central cover of its derived group $G^{\der}$
(which is semisimple) and $\mu=\mu_{G}$ will denote the fundamental
group of $G^{\der}$. There exists a canonical central extension
$$
1\ra\mu\ra\Gtil\ra G^{\der}\ra 1.
$$
Further, if $\lambda\colon G\ra H$ is a homomorphism of reductive
$S$-group schemes, $\widetilde{\lambda}$ will denote the
homomorphism $\Gtil\ra\widetilde{H}$ induced by
$\lambda^{\lbe\der}\colon G^{\der}\ra H^{\der}$.

\begin{proposition} Let $S$ be a non-empty scheme and let $1\ra G _{1}\ra G _{2}
\overset{\lambda}\ra G _{3}\ra 1$ be an exact sequence of reductive
group schemes over $S$, where $G_{1}$ is an $S$-{\rm{torus}}. Then
the map $\widetilde{\lambda}\colon\widetilde{ G }_{2}\ra \widetilde{
G }_{3}$ is an isomorphism.
\end{proposition}
\begin{proof} Note that, since the $S$-torus $G_{1}$ is normal in $G_{2}$ and
$G_{2}$ has connected fibers, $G_{1}$ is central in $G_{2}$
\cite{sga3}, IX.5.5. Now let $C=G_{2}\times_{ G_{3}}\widetilde{ G
}_{3}$. Then $C$ is a reductive $S$-group scheme which fits into an
exact sequence
$$
1\ra G _{1}\ra C \ra\widetilde{ G }_{3}\ra 1.
$$
Since $G_{1}$ is central in $G_{2}$, the preceding sequence is in
fact a central extension. Thus, by Proposition 2.4, there exists an
isomorphism $C \simeq\widetilde{ G }_{3}\times_{S} G _{1}$ which
induces an isomorphism $\widetilde{ C }\overset{\sim}\ra\widetilde{
G }_{3}$. On the other hand, the first projection $C\ra G _{2}$
induces an epimorphism $C^{\der}\!\!\twoheadrightarrow G_{2}^{\der}$
with central kernel, which implies that there exists an isomorphism
$\widetilde{ G }_{2}\overset{\sim}\ra\widetilde{ C }$. The
composition $\widetilde{ G }_{2}\overset{\sim}\ra\widetilde{ C }
\overset{\sim}\ra\widetilde{ G }_{3}$ is the map
$\widetilde{\lambda}$.
\end{proof}

\begin{proposition} Let $S$ be a non-empty scheme and let $1\ra
G_{1}\ra G _{2}\ra G _{3}\ra 1$ be an exact sequence of reductive
group schemes over $S$. Then there exists an induced exact sequence
of simply-connected central covers
$$
1\ra\widetilde{G}_{1}\ra\widetilde{ G }_{2}\ra \widetilde{ G
}_{3}\ra 1.
$$
\end{proposition}
\begin{proof} Let $ G _{3}^{\e\prime}= G _{2}/ G _{1}^{\der}$. Then there exist exact
sequences
$$
1\ra G _{1}^{\der}\ra G _{ 2}^{\der}\ra( G
_{3}^{\e\prime})^{\der}\ra 1
$$
and
$$
1\ra G _{1}^{\tor}\ra G _{3}^{\e\prime}\ra G _{3}\ra 1.
$$
By Proposition 2.8, the first exact sequence induces an exact
sequence $1\ra\widetilde{ G }_{1}\ra\widetilde{ G }_{
2}\ra\widetilde{ G }_{3}^{\e\prime}\ra 1$, and the previous lemma
shows that $\widetilde{ G }_{3}^{\e\prime}$ is isomorphic to
$\widetilde{ G }_{3}$. The result is now clear.
\end{proof}

\begin{corollary} Let $S$ be a non-empty scheme and let $1\ra
 G _{1}\ra G _{2}\ra G _{3}\ra 1$ be an exact sequence of
reductive $S$-group schemes. Then there exists an exact sequence of
$S$-group schemes of multiplicative type
$$
1\ra\mu_{1}\ra\mu_{2}\ra\mu_{3}\ra G _{1}^{\tor}\ra G _{
2}^{\tor}\ra G _{3}^{\tor}\ra 1.
$$
\end{corollary}
\begin{proof} This follows from the proposition by applying the snake lemma
(which is justified in this case) to the exact commutative diagram
$$
\xymatrix{1\ar[r] &\widetilde{ G }_{ 1}\ar[r]\ar[d] &\widetilde{ G
}_{2}\ar[r]
\ar[d] &\widetilde{ G }_{ 3}\ar[r]\ar[d] &1 \\
1\ar[r] & G _{1}\ar[r] & G _{2}\ar[r] & G _{3}\ar[r] &1.}
$$
\end{proof}

\begin{remark} The corollary yields an exact sequence
$$
1\ra G_{3}^{*}\ra G_{2}^{*}\ra G_{1}^{*}\ra \mu_{3}^{\lbe *}\ra
\mu_{2}^{\lbe *}\ra \mu_{1}^{\lbe *}\ra 1
$$
of \'etale sheaves on $S$. When $S=\spec k$ for a separably closed
field $k$, the preceding sequence together with \cite{san}, Lemma
6.9, (ii) and (iii), yields an exact sequence
$$
1\ra G_{3}^{*}(k)\ra G_{2}^{*}(k)\ra G_{1}^{*}(k)\ra\pic\, G_{3}\ra
\pic\, G_{2}\ra\pic\, G_{1}\ra 1
$$
which coincides with \cite{san}, (6.11.4), p.43.
\end{remark}

Let $G$ be a reductive $S$-group scheme and let
$\partial\colon\Gtil\ra G$ be the composition
$\Gtil\twoheadrightarrow G^{\der}\hookrightarrow G$. Thus we have an
exact sequence
\begin{equation}\label{seq3}
1\ra\mu\ra\Gtil\overset{\partial}\ra G\ra G^{\tor}\ra 1.
\end{equation}
There exists a canonical ``conjugation" action of $G$ on $\Gtil$
such that the two-term complex $(\Gtil\be\overset{\partial}\ra\be
G\e)$, where $\Gtil$ and $G$ are placed in degrees $-1$ and $0$,
respectively, is a (left) {\it quasi-abelian} crossed module on
$S_{\e\rm{fl}}$, in the sense of \cite{ga1}, Definition 3.2. See
\cite{br}, Example 1.9, p.28, and \cite{ga1}, Example 2.2(iii). Thus
$\partial$ induces a homomorphism $\partial_{Z}\colon
Z\big(\Gtil\,\big)\ra Z(G)$ and the embedding of crossed modules
$$
\big(Z\big(\Gtil\e\big)\!\overset{\be\partial_{\lbe Z}}
\longrightarrow\! Z(G)\big)\hookrightarrow
(\Gtil\!\overset{\partial}\longrightarrow\! G)
$$
is a quasi-isomorphism \cite{ga1}, Proposition 3.4. In particular,
\eqref{seq3} induces an exact sequence of $S$-group schemes of
multiplicative type
\begin{equation}\label{seq4}
1\ra\mu\ra Z\big(\Gtil\e\big)\overset{\partial_{Z}}\longrightarrow
Z(G) \ra G^{\tor}\ra 1.
\end{equation}

Finally, we will write $\rad(G\e)$ for the radical of $G$, i.e., the maximal torus of $Z(G)$ \cite{sga3}, XXII, definition 4.3.6.
By [op.cit.], Theorem 6.2.1(iii), $\rad(G\e)$ is isogenous to ${\rm{corad}(G)}=G^{\tor}$.

\section{Flasque resolutions}

In this Section we extend J.-L.Colliot-Th\'el\`ene's
construction of flasque resolutions of (connected) reductive algebraic groups over a field contained in \cite{ct} to an
arbitrary admissible base scheme. Some of the proofs from [op.cit.] carry over to this more general setting {\it mutatis mutandis},
while others require substantial modifications.

\smallskip

Let $S$ be an admissible scheme (see Definition 2.1).

\begin{definition} A reductive $S$-group scheme $H$ is called \textit{quasi-trivial} if $H^{\lbe\der}$ is a
simply-connected $S$-group scheme and $H^{\lbe\tor}$ is a
quasi-trivial $S$-torus.
\end{definition}

Thus, by \eqref{seq1}, a quasi-trivial $S$-group scheme is an
extension of a quasi-trivial $S$-torus by a simply-connected
$S$-group scheme. Further, if $H$ is a quasi-trivial $S$-group and
$T\ra S$ is a morphism of admissible schemes, then $H\times_{S}T$ is
a quasi-trivial $T$-group scheme.

\begin{proposition} Let $G$ be a reductive $S$-group scheme
over an admissible scheme $S$. Then there exist a quasi-trivial
$S$-group scheme $H$, a flasque $S$-torus $F$ and a central
extension
$$
1\ra F\ra H \ra G \ra 1.
$$
\end{proposition}

\begin{proof} The following proof is a straightforward generalization of
\cite{ct}, proof of Proposition-Definition 3.1.

By \cite{sga3}, XXII, Theorem 6.2.1, $K:=\rad(G\e)\cap
G^{\der}\subseteq G $ is a finite $S$-group scheme of multiplicative
type contained in $Z(G^{\der})$ ($K$ agrees with the kernel of the
isogeny $\rad( G \e)\ra\rm{corad}( G \e)$ of [loc.cit.],
6.2.1(iii)). Now, by [op.cit.], XXII, 6.2.3, the product in $G$ defines a
faithfully flat morphism $\rad(G\e)\times_{S}\e G^{\der}\ra G$ whose
kernel is equal to the image of $\iota := (1_{\be K },
{\rm{inv}}_{\be K })_{S}\colon K \hookrightarrow K \times_{S} K
\subseteq \rad( G \e)\times_{S} G ^{\der}$, where $1_{\be K }$ and
${\rm{inv}}_{\be K }\colon K\ra K$ are, respectively, the identity
and the inversion morphism on $K$. The map $\rad( G \e)\times_{S} G^{\der}\ra G $ induces a faithfully flat morphism $\rad( G
\e)\times_{S}\widetilde{ G }\ra G $ whose kernel will be denoted by
$\mu_{1}$. Thus, there exists an exact commutative diagram
\begin{equation}
\xymatrix{ 1\ar[r] & \mu_{1}\ar[r]\ar@{->>}[d] & \rad( G
\e)\times_{S}
\widetilde{ G }\ar[r]\ar@{->>}[d] &  G \ar[r]\ar@{=}[d] & 1\\
1\ar[r] & \iota(K)\ar[r] & \rad( G \e)\times_{S} G ^{\der}\ar[r] &
G \ar[r] & 1. } \label{eq_ast}
\end{equation}
The kernel of the middle vertical map above, which is the same as the
kernel of the left-hand vertical map, is canonically isomorphic to
the fundamental group $\mu$ of $ G $. It follows that $\mu_{1}$ is a
finite $S$-group scheme of multiplicative type contained in $\rad( G
\e)\times_{S}\big(\widetilde{ G }\times_{G^{\lbe\der}}\be K\big)$.

Since
$K$ is central in $G^{\der}$ and $\widetilde{ G }\times_{ G ^{\der}}Z( G ^{\der})=Z\big(\widetilde{ G }\e\big)$, $\mu_{1}$ is contained in $\rad( G \e)\times_{S}Z\big(\widetilde{ G
}\e\big)= Z\big(\rad( G \e)\times_{S}\widetilde{ G }\e\big)$. Thus
the rows of \eqref{eq_ast} are {\it central} extensions. Now the
projection $\rad( G \e)\times_{S} Z(\widetilde{ G }\,)\ra
Z(\widetilde{ G }\,)$ yields an embedding $\mu_{1}\hookrightarrow
Z(\widetilde{ G }\,)$. On the other hand, $\rad(G\e)$ is an isotrivial $S$-torus (since $S$ is admissible), so there exists a quasi-trivial $S$-torus $R$ and a surjective
homomorphism $R\twoheadrightarrow\rad(G\e)$ by \cite{cts2}, (1.3.3). Therefore there
exists a faithfully flat homomorphism $R\times_{S}\widetilde{ G }\twoheadrightarrow  G
$ with kernel $M$ (say) and an exact commutative diagram
$$
\xymatrix{ 1\ar[r] & M\ar[r]^(.4){i}\ar@{->>}[d] &R\times_{S}\Gtil
\ar[r]\ar@{->>}[d] & G \ar[r]\ar@{=}[d] &1\\
1\ar[r] &\mu_{1} \ar[r] &\rad( G \e)\times_{S}\widetilde{ G }\ar[r]
& G \ar[r] &1. }
$$
As in \cite{ct}, p.89, lines 2-5, $M$ is an (isotrivial) $S$-group scheme of finite type and of multiplicative type contained in the center of $\e
R\times_{S}\widetilde{ G }$. Next, by \cite{cts2}, (1.3.2), we can
find an exact sequence $1\ra M  \overset{j}\ra F\ra P\ra 1$, where
$F$ is a flasque $S$-torus and $P$ is a quasi-trivial $S$-torus. Let
$H$ be the pushout of $i\colon M \hookrightarrow
R\times_{S}\widetilde{G}$ and $j^{\e\prime}:= {\rm{inv}}_{\lbe F
}\be\circ j\colon M\hookrightarrow F$, i.e., the cokernel of the
central embedding $(i,j)_{S}: M
\hookrightarrow(R\times_{S}\widetilde{G})\times_{S} F$. Then $H$ is a reductive $S$-group scheme
\cite{sga3}, XXII, Corollary 4.3.2, which fits into an exact
commutative diagram
$$
\xymatrix{  & 1\ar[d]&1\ar[d] & &\\
1\ar[r] &  M  \ar[r]^(.4){i} \ar[d]^{j^{\e\prime}} &
R\times_{S}\Gtil
\ar[r]\ar[d] &  G \ar[r]\ar@{=}[d] & 1\\
1\ar[r] & F \ar[r]\ar[d] &  H \ar[r]\ar[d] &  G \ar[r] & 1\\
& P \ar[d]\ar@{=}[r]&P\ar[d] & &\\
& 1&1 & &}
$$
with $F$ central in $ H $. Let $\varepsilon\colon S\ra R$ be the
unit section of $R$ and let $\psi=(\varepsilon,
{\text{Id}}_{\lbe\widetilde{ G }})_{S}\colon S\times_{S}\widetilde{
G }\ra R\times_{S}\widetilde{ G }$. Then there exists an
$S$-morphism
$$
\widetilde{ G }\simeq S\times_{S}\widetilde{ G
}\overset{\psi}\hookrightarrow R\times_{S}\widetilde{ G
}\hookrightarrow H
$$
which embeds $\widetilde{ G }$ as a normal subgroup scheme of $ H $.
The corresponding quotient $ H /\e\widetilde{ G }$ fits into an
exact commutative diagram
$$
\xymatrix{&& 1\ar[d]&1\ar[d] &\\
1\ar[r] &\widetilde{ G }\ar[r]\ar@{=}[d] &R\times_{S}\widetilde{ G }\ar[r]\ar[d] & R\ar[r]\ar[d] &1\\
1\ar[r] &\widetilde{ G }\ar[r] & H \ar[r]\ar[d] & H /\,\widetilde{ G }\ar[r]\ar[d] &1\\
& &P \ar[d]\ar@{=}[r]& P\ar[d] & \\
 && 1&1 & }
$$
By \cite{sga3}, XXII, 6.3.1 and 6.2.1(iv), $H^{\der}\simeq
(\Gtil\,)^{\der}=\widetilde{ G }$ is simply-connected. Further, $ H
/\e\widetilde{ G }\simeq H ^{\tor}$ is an extension of the
quasi-trivial $S$-torus $P$ by the quasi-trivial $S$-torus $R $. Now Lemma 2.2 shows that
$$
H^{\tor}\simeq R\times_{S} P
$$
is a quasi-trivial $S$-torus, which completes the proof.
\end{proof}

\begin{remarks}\indent
\begin{enumerate}

\item[(a)] Let $S^{\e\prime}\ra S$ be a finite connected
Galois cover of $S$ with Galois group $\varDelta$ which splits both
${\rm{rad}}(G)$ (or, equivalently, ${\rm{corad}}(G)=G^{\tor}$) and
$Z\big(\Gtil\big)$. Then the proof of the proposition shows that
there exists a flasque resolution $1\ra F\ra H \ra G \ra 1$ of $G$
such that $F$ is split by $S^{\e\prime}\ra S$. Consequently, if
$\varDelta$ is {\it metacyclic} (i.e., if every Sylow subgroup of
$\varDelta$ is cyclic), then $F$ is {\it invertible} \cite{cts1},
Proposition 2, p.184, and we conclude that $G$ admits an invertible
resolution.

\item[(b)] Let $G$ and $S$ be as in the statement of the
proposition, let $1\ra F\ra H \ra G \ra 1$ be a flasque resolution of $G$ and let $T\ra S$ be a morphism of admissible schemes. Then $1\ra F\times_{S}T\ra H\times_{S}T \ra G\times_{S}T \ra 1$ is a flasque resolution of the reductive $T$-group scheme $G\times_{S}T$. See \cite{cts2}, Proposition 1.4.
\end{enumerate}

\end{remarks}

Let $S$ and $G$ be as in the proposition and let
\begin{equation}\label{fr}
1\ra F\ra H \ra G \ra 1
\end{equation}
be a flasque resolution of $G$. Set $R= H^{\tor}$. Then the proof of
the proposition (see also \cite{ct}, p.93) shows that \eqref{fr}
induces a so-called ``fundamental" exact commutative diagram
\begin{equation}\label{fund}
\xymatrix{ &1\ar[d] &1\ar[d] &1\ar[d] &\\
1\ar[r] &\mu\ar[r]\ar[d] &\widetilde{ G }\ar[r]\ar[d] & G ^{\der}\ar[r]\ar[d] &1\\
1\ar[r] & F\ar[r]\ar[d] & H \ar[r]\ar[d] & G \ar[r]\ar[d] &1\\
1\ar[r] &M\ar[r]\ar[d] &R\ar[r]\ar[d] & G^{\tor}\ar[r]\ar[d] &1,\\
&1 &1 &1 & }
\end{equation}
where $M=F/\mu$. The above diagram immediately yields an exact
sequence
\begin{equation}\label{sec}
1\ra\mu\ra F\ra R\ra G^{\tor}\ra 1.
\end{equation}
Further, \eqref{fund} induces an exact commutative diagram of
$S$-group schemes of multiplicative type
\begin{equation}\label{zfund}
\xymatrix{ &1\ar[d] &1\ar[d] &1\ar[d] &\\
1\ar[r] &\mu\ar[r]\ar[d] &Z\big(\widetilde{ G }\e\big)
\ar@{-->}[rd]^{\partial_{Z}}\ar[r]\ar[d] & Z(G^{\der})\ar[r]\ar[d] &1\\
1\ar[r] & F\ar@{-->}[rd]\ar[r]\ar[d] & H\be\times_{G}\be Z(G)
\ar[r]\ar[d] &
Z(G) \ar[r]\ar[d] &1\\
1\ar[r] & M\ar[r]\ar[d] & R\ar[r]\ar[d] & G ^{\tor}\ar[r]\ar[d] &1.\\
&1 &1 &1 & }
\end{equation}

\begin{proposition} The complexes
$\big(Z\big(\Gtil\e\big)\!\be\overset{\partial_{Z}}
\longrightarrow\! Z(G)\e\big)$ and $(F\be\ra\be R\e)$ are
canonically quasi-isomorphic.
\end{proposition}
\begin{proof} The proof is formally the same as the proof of \cite{ct},
Proposition A.1, p.124. Indeed, \eqref{zfund} induces an exact
commutative diagram
$$
\xymatrix{1\ar[r] &\mu\ar[r]\ar[d]^{=} & Z\big(\Gtil\e\big)
\ar[r]\ar@{^{(}->}[d] & Z(G)\ar[r]\ar@{^{(}->}[d]&G^{\tor}\ar[d]^{=}\ar[r] &1\\
1\ar[r] & \mu\ar[r] & H\be\times_{G}\be Z(G) \ar[r]& R\oplus Z(G) \ar[r]&G^{\tor}\ar[r] &1\\
1\ar[r] & \mu\ar[r]\ar[u]_{=} & F^{\phantom{i}}\ar[r]\ar@{_{(}->}[u]
& R^{\phantom{i}}\ar[r]\ar@{_{(}->}[u]& G^{\tor}\ar[r]\ar[u]_{=}
&1,}
$$
where the top row is \eqref{seq4} and the bottom row is \eqref{sec}.
It follows that the complexes in the statement are both
quasi-isomorphic to the complex $(H\be\times_{G}\be Z(G)\ra R\oplus
Z(G))$. For more details, see [op.cit.], p.126.
\end{proof}

\smallskip

We now turn to the problem of comparing two flasque resolutions of a
given reductive group scheme $G$. We begin with

\begin{lemma} Let $S$ be an admissible scheme. Then
every central extension of a quasi-trivial $S$-group scheme by a
flasque $S$-torus is split.
\end{lemma}
\begin{proof} Let $H$ be a quasi-trivial $S$-group scheme, let $F$ be a
flasque $S$-torus and let $1\ra F\ra E\ra H\ra 1$ be a central
extension. Applying Remark 2.5 to the central extension $1\ra F\ra
E\times_{H}H^{\der}\ra H^{\der}\ra 1$, we conclude that the
inclusion $H^{\der}\hookrightarrow H$ induces a morphism
$H^{\der}\ra E$ which maps $H^{\der}$ isomorphically onto
$E^{\der}$. It follows that there exists an exact commutative
diagram
$$
\xymatrix{1\ar[r]&F\ar@{=}[d]\ar[r]&E\ar[d]\ar[r]&H
\ar[d]^{c}\ar[r]&1\\
1\ar[r]& F\ar[r]& E^{\tor}\ar[r]& H^{\tor}\ar[r]&1,}
$$
where $c$ is the canonical map. The above diagram induces another  exact commutative diagram
$$
\xymatrix{1\ar[r]&F\ar@{=}[d]\ar[r]&E\ar[d]\ar[r]&H
\ar@{=}[d]\ar[r]&1\\
1\ar[r]& F\ar[r]& E^{\lbe\tor}\times_{H^{\be\tor}}H\ar[r]& H\ar[r]&1,}
$$
which shows that $E\simeq E^{\lbe\tor}\times_{H^{\be\tor}}H$. Since $E^{\lbe\tor}\simeq F\times_{S}H^{\lbe\tor}$ by Lemma 2.2, the result follows.
\end{proof}

We can now compare two flasque resolutions of a given reductive
$S$-group scheme $G$.

\begin{proposition} Let $S$ be an admissible scheme and let
$ G $ be a reductive group scheme over $S$. Let $1\ra F\ra H \ra G
\ra 1$ and $1\ra F_{1}\ra H _{1}\ra G \ra 1$ be two flasque
resolutions of $G $. Set $R= H ^{\tor}$ and $R_{1}=H^{\tor}_{1}$.
Then
\begin{enumerate}
\item[(i)] There exists an isomorphism
$ F \times_{S} H _{1}\simeq  F _{1}\times_{S} H $.
\item[(ii)] The simply-connected $S$-group schemes
$H^{\der}$ and $H^{\der}_{1}$ are isomorphic.
\item[(iii)] There exists an isomorphism of $S$-tori
$F\times_{S} R_{1}\simeq  F_{1}\times_{S} R$.
\item[(iv)] There exists a canonical isomorphism of
twisted-constant $S$-group schemes
$$
\cok[F_{*}\ra R_{*}]\simeq\cok[F_{1*}\ra R_{1*}].
$$
\end{enumerate}
\end{proposition}
\begin{proof} The given flasque resolutions induce central extensions
\begin{equation}
1\ra F \ra H \times_{ G }\be H _{1}\ra H _{1}\ra 1 \label{eq_ast3}
\end{equation}
and
\begin{equation}
1\ra F _{1}\ra H _{1}\times_{ G }\be H \ra H \ra 1. \label{eq_ast4}
\end{equation}
Since both $H$ and $H_{1}$ are quasi-trivial, the above extensions
are split by the previous lemma. It follows that $ F \times_{S} H
_{1}\simeq  H \times_{ G } H _{1} \simeq  F _{1}\times_{S} H $,
which proves (i). Further (see Remark 2.5), \eqref{eq_ast3} and
\eqref{eq_ast4} show that $ H^{\der}_{1}\simeq ( H \times_{G} H
_{1})^{\der}\simeq H ^{\der}$, i.e., (ii) holds. Now the diagram
$$
\xymatrix{   && 1\ar[d]&1\ar[d] & \\
&  & ( H \times_{ G }\be H _{1})^{\der}\ar[d] \ar[r]^(.6){ \sim} & H ^{\der}_{1}\ar[d] &\\
1 \ar[r] &  F  \ar[r] &  H \times_{ G } H _{1} \ar[r] & H _{1}
\ar[r] & 1 }
$$
yields an exact sequence of $S$-tori
$$
1\ra F \ra( H \times_{ G } H _{1})^{\tor}\ra R _{1}\ra 1.
$$
By Lemma 2.2 this sequence splits. Thus $(H\times_{G} H_{\e
1})^{\tor}\simeq F \times_{S} R_{\e 1}$. Similarly, $( H _{\e
1}\times_{ G } H )^{\tor}\simeq F _{\e 1}\times_{S} R $ and (iii)
follows. Finally, applying the kernel-cokernel exact sequence
\cite{adt}, Proposition I.0.24, p.16, to
$$
F_{*}\ra F_{*}\oplus R_{\e 1*}\simeq F_{\e 1*}\oplus R_{*}
\twoheadrightarrow R_{*},
$$
we conclude that the morphisms $ F_{*}\ra R_{*}$ and $ F_{\e 1*}\ra
R_{\e 1*}$ (induced by the compositions $ F \ra H \ra R $ and $F_{\e
1}\ra H_{\e 1}\ra R_{\e 1}$, respectively) have canonically
isomorphic cokernels. This proves (iv).
\end{proof}

Part (iv) of the above proposition, together with \cite{ct}, Remark
3.2.2 (which carries over to the present setting {\it mutatis
mutandis}), motivates the following definition.

\begin{definition} Let $G$ be a reductive group scheme over an admissible
scheme $S$. Let $1\ra F \ra H \ra G \ra 1$ be a flasque resolution
of $G$ and set $R=H^{\tor}$. Then the twisted-constant $S$-group scheme
$$
\pi_{1}(G\e):=\cok[F_{\be *}\ra R_{*}]
$$
is independent (up to isomorphism) of the  chosen flasque resolution of $G$. It is called the \textit{algebraic fundamental group} of $G$.
\end{definition}

\begin{remark} Since $\mu_{*}=0$, sequence \eqref{sec} shows that
$F_{\be *}\ra R_{*}$ is injective. Thus, by the definition of
$\pi_{1}( G \e)$, there exists an exact sequence of \'etale, finitely generated twisted-constant $S$-group schemes
$$
1\ra F _{\be *}\ra R _{*}\ra\pi_{1}( G \e)\ra 1.
$$
\end{remark}

\smallskip

If $\mu$ is a finite $S$-group scheme of multiplicative type, set
$$
\mu(-1):=\hom_{S\lbe\text{-gr}}(\mu^{*},(\bq/\bz\be)_{S}).
$$

\begin{proposition} Let $S$ be an admissible scheme and let $G$ be a reductive group
scheme over $S$ with fundamental group $\mu$. Then there exists an
exact sequence of \'etale, finitely generated twisted-constant $S$-group schemes
$$
1\ra\mu(-1)\ra\pi_{1}( G \e)\ra( G ^{\tor})_{*}\ra 1.
$$
\end{proposition}
\begin{proof} The proof of \cite{ct}, Proposition 6.4, carries
over to the present setting {\it mutatis mutandis}, using
\eqref{sec} (but note that the first exact sequence in [loc. cit.]
has been dualized incorrectly).
\end{proof}

\begin{remark} If $G$ is {\it semisimple}, then $G^{\tor}=0$ and
the proposition shows that $\pi_{1}(G\e)=\mu(-1)$.  Thus the
algebraic fundamental group of $G$ is the Pontryagin dual of the
\'etale fundamental group of $G$. On the other hand, if $G$ is a
torus (so that $\mu=0$ and $ G^{\tor}=G$), the proposition yields
the equality $\pi_{1}\lbe(G\e)=G_{*}$.
\end{remark}

\begin{proposition} Let $S$ be an admissible scheme.
\begin{enumerate}
\item[(i)] Every homomorphism of reductive $S$-group schemes
$\lambda\colon G _{ 1}\ra G _{ 2}$ induces a homomorphism of
algebraic fundamental $S$-group schemes $\lambda_{*}\colon\pi_{1}( G
_{ 1})\ra\pi_{1}( G _{ 2})$.
\item[(ii)] If $\kappa\colon G _{ 1}\ra G _{ 2}$ and
$\lambda\colon G _{ 2}\ra G _{ 3}$ are homomorphisms of reductive
$S$-group schemes, then
$(\lambda\lbe\circ\lbe\kappa)_{*}=\lambda_{*}\lbe\circ\lbe\kappa_{*}$.
\end{enumerate}
\end{proposition}
\begin{proof} The following proof of (i) is very similar to the proof of
\cite{ct}, Proposition 6.6(i). For $i=1,2$, let $1\ra F _{i}\ra H
_{i}\ra G _{i}\ra 1$ be a flasque resolution of $ G_{i}$. Then, as
seen in the proof of Proposition 3.6(i), there exists an isomorphism
$H_{1}\times_{G_{2}} H_{2}\simeq H_{1}\times_{S} F_{2}$. Thus
$\lambda\colon G_{1}\ra G_{2}$ induces a homomorphism $H_{1}\ra
H_{2}$ which fits into an exact commutative diagram
$$
\xymatrix{1\ar[r]& F _{1}\ar[r]\ar[d]& H_{1}\ar[r]\ar[d]& G_{
1}\ar[r]
\ar[d]^{\lambda}&1\\
1\ar[r]& F_{2}\ar[r]& H_{2}\ar[r]& G_{2}\ar[r]&1.}
$$
Any two such liftings $H_{\e 1}\ra H_{\e 2}$ of $\lambda$  differ by
an $S$-morphism $H_{1}\ra F_{2}$ which is, in fact, an
$S$-homomorphism by the argument in \cite{ray}, proof of Corollary
VII.1.2, p.103. The rest of the proof of (i) is as in \cite{ct},
proof of Proposition 6.6(i), using Proposition 3.6 above. The
homomorphism $\lambda_{\le *}\colon\pi_{1}( G _{ 1})\ra\pi_{1}( G _{
2})$ thus defined is {\it independent of the choice of liftings}. To
prove (ii), one chooses as lifting of $\lambda\circ\kappa\colon
 G _{1}\ra G _{ 3}$ the composition of a lifting of $\lambda$ and a lifting
of $\kappa$, and uses the independence of choices just mentioned.
\end{proof}

It follows from part (ii) of the previous proposition that an exact
sequence
$$
1\longrightarrow G _{1}\overset{\kappa}\longrightarrow G
_{2}\overset{\lambda} \longrightarrow G _{3}\longrightarrow 1
$$
of reductive group schemes over an admissible scheme $S$ induces a
{\it complex} of algebraic fundamental $S$-group schemes
$$
1\longrightarrow\pi_{1}( G _{1})\overset{\kappa_{\lbe
*}}\longrightarrow\pi_{1}( G _{2})\overset{\lambda_{\lbe
*}}\longrightarrow\pi_{1}( G _{3})\longrightarrow 1.
$$
We will show that the preceding complex is in fact {\it exact}, by
proving first that it is exact when $G_{1}$ is either a semisimple
$S$-group scheme or an $S$-torus, and then using these particular
cases to obtain the general result (see Theorem 3.14 below).

\begin{lemma} Let $S$ be an admissible scheme and let
$1\ra G _{ 1}\ra G _{ 2}\ra G _{3}\ra 1$ be an exact sequence of
reductive $S$-group schemes, where $ G _{1}$ is {\rm{semisimple}}.
Then
$$
1\ra\pi_{1}( G _{ 1})\ra\pi_{1}( G _{ 2})\ra\pi_{1}( G _{ 3}) \ra 1
$$
is exact.
\end{lemma}
\begin{proof} Since $ G _{1}= G _{1}^{\der}\subset G _{2}^{\der}$, the given
sequence induces an exact sequence of semisimple $S$-group schemes
$1\ra G _{ 1}\ra G _{ 2}^{\der}\ra G _{ 3}^{\der}\ra 1$. Now, by
Proposition 2.8, there exists an exact commutative diagram
$$
\xymatrix{1\ar[r]&\Gtil_{ 1}\ar@{->>}[d]\ar[r]&\widetilde{ G }_{
2}\ar@{->>}[d]\ar[r]&\widetilde{ G }_{3}\ar@{->>}[d]
\ar[r]&1\\
1\ar[r]& G _{1}\ar[r]& G _{ 2}^{\der}\ar[r]&G _{ 3}^{\der}\ar[r]&1,
}
$$
which yields an exact sequence
$$
1\ra\mu_{ 1}(-1)\ra\mu_{ 2}(-1)\ra\mu_{ 3}(-1)\ra 1.
$$
On the other hand, there exists an exact sequence of
twisted-constant $S$-group schemes $1\ra G_{3}^{*}\ra G_{2}^{*}\ra
G_{1}^{*}=1$, and it follows that the canonical map $( G
_{2}^{\tor})_{*}\ra( G _{3}^{\tor})_{*}$ is an isomorphism. Thus
there exists an exact commutative diagram
$$
\xymatrix{1\ar[r]&\mu_{ 2}(-1)\ar@{->>}[d]\ar[r]&\pi_{1}( G _{ 2})
\ar[d]\ar[r]&( G _{2}^{\tor})_{*}\ar@{=}[d]\ar[r]&1\\
1\ar[r]&\mu_{ 3}(-1)\ar[r]&\pi_{1}( G _{ 3})\ar[r]&( G _{
3}^{\tor})_{*}\ar[r]&1,}
$$
which yields an exact sequence
$$
1\ra\mu_{ 1}(-1)\ra\pi_{1}( G _{ 2})\ra\pi_{1}( G _{ 3})\ra 1.
$$
Since $\mu_{1}(-1)=\pi_{1}( G _{1})$ by Remark 3.10, the proof is
complete.
\end{proof}

\begin{lemma} Let $S$ be an admissible scheme and let
$1\ra G _{1}\ra G _{2}\ra G _{3}\ra 1$ be an exact sequence of
reductive group schemes over $S$, where $ G _{1}$ is an
$S$-{\rm{torus}}. Then
$$
1\ra\pi_{1}( G _{1})\ra\pi_{1}( G _{2})\ra\pi_{1}( G _{3})\ra 1
$$
is exact.
\end{lemma}
\begin{proof} By Proposition 2.9, the map
$\widetilde{ G }_{ 2}\ra\widetilde{ G }_{3}$ induced by $ G _{2}\ra
G _{3}$ is an isomorphism. Now set $\mu^{\e\prime}= G _{ 1}\cap G _{
2}^{\der}=\krn[\e G _{2}^{\der}\ra G _{3}^{\der}\e]$ and consider
$$
\xymatrix{1\ar[r] &\mu_{ 2}\ar[r]\ar[d] &\widetilde{ G }_{ 2}\ar[r]\ar[d]^{\simeq} & G _{2}^{\der}\ar[r]\ar[d] & 1\\
1\ar[r] &\mu_{3}\ar[r] & \widetilde{ G }_{3}\ar[r] & G _{
3}^{\der}\ar[r] &1.}
$$
The above diagram yields an exact sequence
\begin{equation}\label{mus}
1\ra\mu_{2}(-1)\ra\mu_{3}(-1)\ra\mu^{\e\prime}(-1)\ra 1.
\end{equation}
On the other hand, the diagram
$$
\xymatrix{1\ar[r] &\mu^{\e\prime}\ar[r]\ar@{^{(}->}[d] & G _{ 2}^{\der}\ar[r]\ar@{^{(}->}[d] & G _{ 3}^{\der}\ar[r]\ar@{^{(}->}[d] & 1\\
1\ar[r] & G _{ 1}\ar[r] &  G _{ 2}\ar[r] & G _{ 3}\ar[r] &1}
$$
yields an exact sequence of $S$-group schemes of multiplicative type
$$
1\ra\mu^{\e\prime}\ra G _{1}\ra G _{2}^{\tor}\ra G _{3}^{\tor} \ra
1.
$$
The preceding sequence induces exact sequences
\begin{equation}\label{tor2}
1\ra( G _{1}/\mu^{\e\prime})_{*}\ra( G _{ 2}^{\tor})_{*}\ra( G _{
3}^{\tor})_{*}\ra 1
\end{equation}
and
\begin{equation}\label{tor3}
1\ra (G _{1})_{*}\ra( G _{1}/\mu^{\e\prime})_{*}\ra
\mu^{\e\prime}(-1)\ra 1.
\end{equation}
Now the diagram
$$
\xymatrix{0\ar[r] &\mu_{ 2}(-1)\ar[r]\ar[d] &\pi_{1}( G _{ 2})\ar[r]\ar[d] &( G _{ 2}^{\tor})_{*}\ar[d]\ar[r] &1\\
0\ar[r] &\mu_{ 3}(-1)\ar[r] &\pi_{1}( G _{ 3})\ar[r] &( G _{
3}^{\tor})_{*}\ar[r] &1,}
$$
together with \eqref{mus} and \eqref{tor2}, yields an exact sequence
$$
\krn\be[\pi_{1}( G _{2})\ra\pi_{1}( G _{ 3})]\hookrightarrow( G _{
1}/\mu^{\e\prime})_{*}\ra\mu^{\e\prime}(-1)\twoheadrightarrow
\cok\be[\pi_{1}( G _{ 2})\ra\pi_{1}( G _{ 3})].
$$
It now follows from the above sequence and \eqref{tor3} that
$\krn\be[\pi_{1}( G _{2})\ra\pi_{1}( G _{3})]\simeq G _{
1*}=\pi_{1}( G _{1})$ and $\cok\be[\pi_{1}( G _{2})\ra\pi_{1}( G
_{3})]=0$, as claimed.
\end{proof}

\begin{theorem} Let $S$ be an admissible scheme and let
$1\ra G _{ 1}\ra G _{ 2}\ra G _{ 3}\ra 1$ be an exact sequence of
reductive group schemes over $S$. Then the induced sequence
$$
1\ra\pi_{1}( G _{1})\ra\pi_{1}( G _{2})\ra\pi_{1}( G _{3})\ra 1
$$
is exact.
\end{theorem}
\begin{proof} Let $ G _{3}^{\e\prime}= G _{2}/ G _{1}^{\der}$. Then there exist
exact sequences
$$
1\ra G _{1}^{\der}\ra G _{2}\ra G _{3}^{\e\prime}\ra 1
$$
and
$$
1\ra G _{1}^{\tor}\ra G _{3}^{\e\prime}\ra G _{3}\ra 1.
$$
Therefore, by Lemmas 3.12 and 3.13, there exists an exact
commutative diagram
$$
\xymatrix{&&1\ar[d]\\
&& \pi_{1}( G _{ 1}^{\der})\ar[d] & \\
&&\pi_{1}( G _{ 2})\ar[d]\ar@{..>>}[rd] &\\
1\ar[r]&\pi_{1}( G _{ 1}^{\tor})\ar[r] &\pi_{1}( G _{3}^{\e\prime})
\ar[d]\ar[r] &\pi_{1}( G _{3})\ar[r]&1\\
&&1}
$$
which shows that the dotted arrow is a surjection. Further, the
above diagram yields the bottom row of the exact commutative diagram
$$
\xymatrix{1\ar[r] &\mu_{ 1}(-1)\ar[r]\ar@{=}[d] &\pi_{1}( G _{ 1})\ar[r]\ar[d] &( G _{ 1}^{\tor})_{*}\ar[r]\ar@{=}[d] &1 \\
1\ar[r] &\pi_{1}( G _{ 1}^{\der})\ar[r] &\krn[\pi_{1}( G _{
2})\ra\pi_{1}( G _{ 3})]\ar[r] &\pi_{1}( G _{ 1}^{\tor})\ar[r] &1.}
$$
The theorem is now clear.
\end{proof}

\begin{remark} The above result extends \cite{ct}, Proposition 6.8, from the
case of a base of the form $\spec k$, where $k$ is a field of
characteristic zero, to an arbitrary admissible base scheme.
\end{remark}

\section{Abelian cohomology and flasque resolutions}

In this Section we give some applications of flasque resolutions.
See \cite{ga2} for additional applications.

Let $S$ be a non-empty scheme. We will write $S_{\e\rm{fl}}$
(respectively, $S_{\e\rm{\acute{e}t}}$) for the small fppf
(respectively, \'etale) site over $S$. If $F_{1}$ and $F_{2}$ are
abelian sheaves on $S_{\e\rm{fl}}$ (regarded as complexes
concentrated in degree 0), $F_{1}\lbe\otimes^{\e\mathbf{L}}\lbe
F_{2}$ (respectively, $\rhom\e(F_{1},F_{2})$) will denote the total
tensor product (respectively, right derived hom functor) of $F_{1}$
and $F_{2}$ in the derived category of the category of abelian
sheaves on $S_{\e\rm{fl}}$.

If $\tau=\rm{fl}$ or $\rm{\acute{e}t}$, $G$ is an $S$-group scheme
and $i=0$ or $1$, $H^{\le i}(S_{\e\tau},G)$ will denote the $i$-th
cohomology set of the sheaf on $S_{\e\tau}$ represented by $G$. If
$G$ is commutative, these cohomology sets are in fact abelian groups
and are defined for every $i\geq 0$. When $G$ is smooth, the
canonical map $H^{\le i}(S_{\rm{\acute{e}t}},G)\ra H^{\e
i}(S_{\e\rm{fl}},G)$ is bijective \cite{mi1}, Remark III.4.8(a),
p.123. In this case, the preceding sets will be identified.

\medskip

Let $G$ be a reductive group scheme over $S$ and recall from Section
2 the complex
$(Z\big(\Gtil\,\big)\overset{\partial_{Z}}\longrightarrow Z(G\e))$
determined by $G$. For any integer $i\geq -1$, the $i$-th
\textit{abelian (flat) cohomology group of $G$} is by definition the
hypercohomology group
\begin{equation}\label{abcohom}
H^{\le i}_{\rm{ab}}(S_{\e\rm{fl}}, G\e)={\bh}^{\e
i}\big(S_{\e\rm{fl}},
Z\big(\Gtil\,\big)\overset{\partial_{Z}}\longrightarrow Z(G\e)).
\end{equation}
On the other hand, the $i$-th \textit{dual abelian cohomology group of $\e G$} is the group
\begin{equation}\label{dabcohom}
H^{\le i}_{\rm{ab}}(S_{\e\rm{\acute{e}t}},G^{*})={\bh}^{\e
i}\big(S_{\e\rm{\acute{e}t}},Z(G\e)^{*}\!\overset{\partial_{\lbe
Z}^{*}}\longrightarrow Z\big(\Gtil\,\big)^{\lbe *}\e\big).
\end{equation}
These groups arise in the study of the cohomology sets of $G$ over
$S$ and of various arithmetic objects associated to $G$. See
\cite{bor,ga1,ga2} for more details. Note that, since the Cartier
dual of an $S$-group scheme of multiplicative type is \'etale (in
particular, smooth) over $S$, the groups \eqref{dabcohom} coincide
with the flat hypercohomology groups ${\bh}^{\e
i}\big(S_{\e\rm{fl}},Z(G\e)^{*}\ra Z\big(\Gtil\,\big)^{\lbe
*}\e\big)$. If $S=\spec K$, where $K$ is a field, $H^{\le
i}_{\rm{ab}}(K,G^{*})$ will denote $H^{\le
i}_{\rm{ab}}(S_{\e\rm{\acute{e}t}},G^{*})$. Clearly, if $\kb$ is a
fixed separable algebraic closure of $K$ and $\g=\text{Gal}(\kb/K)$,
$$
H^{\le i}_{\rm{ab}}(K,G^{*})={\bh}^{\e i}\big(\g,Z\big(G\big)
^{*}\be\big(\e\kb\e\big)\ra
Z\big(\Gtil\e\big)^{*}\be\big(\e\kb\e\big)\e\big)
$$
(Galois hypercohomology). Now, by \eqref{seq4}, there exist exact
sequences
\begin{equation}\label{kamb}
\dots\ra H^{\e i-1}(S_{\e\rm{\acute{e}t}},G^{\tor}\e)\ra H^{\le
i+1}(S_{\e\rm{fl}},\mu\e)\ra H^{\e i}_{\rm{ab}}(S_{\e\rm{fl}},G)\ra
H^{\e i}(S_{\e\rm{\acute{e}t}},G^{\tor}\e)\ra\dots.
\end{equation}
and
\begin{equation}\label{kamb*}
\dots\ra H^{\e i-1}(S_{\e\rm{\acute{e}t}},\mu^{\lbe *}\e)\ra H^{\le
i+1}(S_{\e\rm{\acute{e}t}},G^{\tor*}\e)\ra H^{\e
i}_{\rm{ab}}(S_{\e\rm{\acute{e}t}},G^{*})\ra H^{\e
i}(S_{\e\rm{\acute{e}t}},\mu^{\lbe *}\e)\ra\dots.
\end{equation}

\begin{examples}\indent

\begin{enumerate}

\item[(a)] If $G$ is semisimple, i.e., $G^{\tor}=0$, then $H^{\e i}_{\rm{ab}}
(S_{\e\rm{fl}},G)=H^{\le i+1}(S_{\e\rm{fl}},\mu\e)$ and $H^{\e
i}_{\rm{ab}}(S_{\e\rm{\acute{e}t}},G^{*})=H^{\e
i}(S_{\e\rm{\acute{e}t}},\mu^{\lbe *})$.
\item[(b)] If $G$ has trivial fundamental group, i.e., $\mu=0$, then
$H^{\e i}_{\rm{ab}}(S_{\e\rm{fl}},G)=H^{\e
i}(S_{\e\rm{\acute{e}t}},G^{\tor}\e)$ and $H^{\e
i}_{\rm{ab}}(S_{\e\rm{\acute{e}t}},G^{*})=H^{\le
i+1}(S_{\e\rm{\acute{e}t}},G^{\tor*}\e)$.
\end{enumerate}

\end{examples}

The following result is an immediate consequence of Proposition 3.4.

\begin{proposition} Let $S$ be an admissible scheme, let $G$ be a reductive group scheme over $S$ and let $1\ra F\ra H\ra G\ra 1$ be a flasque
resolution of $G$. Set $R=H^{\lbe\tor}$. Then the given resolution
defines isomorphisms $H^{\le i}_{\rm{ab}}(S_{\e\rm{fl}}, G\e)\simeq
{\bh}^{\e i}\big(S_{\e\rm{fl}},F\ra R\e)$ and $H^{\le
i}_{\rm{ab}}(S_{\e\rm{\acute{e}t}}, G^{*}\e)\simeq {\bh}^{\e
i}\big(S_{\e\rm{\acute{e}t}},R^{*}\ra F^{*})$. Further, the given
resolution defines exact sequences
$$
\dots\ra H^{\le i}(S_{\e\rm{\acute{e}t}},F\e) \ra H^{\le
i}(S_{\e\rm{\acute{e}t}},R\e)\ra H^{\le
i}_{\rm{ab}}(S_{\e\rm{fl}},G\e)\ra H^{\le
i+1}(S_{\e\rm{\acute{e}t}},F\e)\ra\dots
$$
and
$$
\dots\ra H^{\le i}(S_{\e\rm{\acute{e}t}},R^{*}\e) \ra H^{\le
i}(S_{\e\rm{\acute{e}t}},F^{*}\e)\ra H^{\le
i}_{\rm{ab}}(S_{\e\rm{\acute{e}t}},G^{*}\e)\ra H^{\le
i+1}(S_{\e\rm{\acute{e}t}},R^{*}\e)\ra\dots.\qed
$$
\end{proposition}

\begin{corollary} Let $G$ be a reductive group scheme over an admissible scheme $S$.
Then, for every integer $i\geq -1$, there exist isomorphisms
$$
H^{\le i}_{\rm{ab}}(S_{\e\rm{fl}}, G \e)\simeq{\bh}^{\e i}
(S_{\e\rm{fl}},\pi_{1}\lbe(G\e)\lbe\otimes^{\e\mathbf{L}}\be
\bg_{m,S})
$$
and
$$
H^{\le i}_{\rm{ab}}(S_{\e\rm{\acute{e}t}},G^{*}\e)\simeq{\bh}^{\e i}
(S_{\e\rm{\acute{e}t}},\rhom(\pi_{1}\lbe(G\e),\bz_{\! S})).
$$
\end{corollary}
\begin{proof} Let $1\ra F\ra H \ra G \ra 1$ be a flasque resolution of $G$ and
set $R=H^{\tor}$. By the proposition, the given resolution induces
isomorphisms $H^{\le i}_{\rm{ab}}(S_{\e\rm{fl}}, G\e)\simeq
{\bh}^{\e i}\big(S_{\e\rm{fl}},F\ra R\e)$ and $H^{\le
i}_{\rm{ab}}(S_{\e\rm{\acute{e}t}},G^{*}\e)\simeq {\bh}^{\e
i}\big(S_{\e\rm{\acute{e}t}},R^{*}\ra F^{*}\e)$. On the other hand,
by \eqref{bas0}, \eqref{bas} and Remark 3.8, $(F\ra
R)=(F_{*}\otimes_{\e\bz_{\! S}}\bg_{m,S}\ra R_{*}\otimes_{\e\bz_{\!
S}}\bg_{m,S})$ is quasi-isomorphic to $\pi_{1}\lbe( G
\e)\lbe\otimes^{\e\mathbf{L}} \bg_{m,S}$ and $(R^{*}\ra F^{*}\e)=
(\hom_{S\lbe\text{-gr}}(R_{*},\bz_{\! S})\ra
\hom_{S\lbe\text{-gr}}(F_{*},\bz_{\! S}))$ is quasi-isomorphic to
$\rhom(\pi_{1}\lbe(G\e),\bz_{\! S})$. This yields the desired
isomorphisms.
\end{proof}

\begin{proposition} Let $S$ be an admissible scheme and let
$1\ra G_{1}\ra G_{2}\ra G_{3}\ra 1$ be an exact sequence of
reductive group schemes over $S$. Then there exist exact sequences
of abelian groups
$$
\dots\ra H^{\le i}_{\rm{ab}}(S_{\e\rm{fl}},G_{1}\e)\ra H^{\le
i}_{\rm{ab}}(S_{\e\rm{fl}},G_{2}\e)\ra H^{\le
i}_{\rm{ab}}(S_{\e\rm{fl}}, G_{3}\e)\ra H^{\le
i+1}_{\rm{ab}}(S_{\e\rm{fl}},G_{1}\e)\ra\dots
$$
and
$$
\dots\ra H^{\le i}_{\rm{ab}}(S_{\e\rm{\acute{e}t}},G_{3}^{*}\e)\ra
H^{\le i}_{\rm{ab}}(S_{\e\rm{\acute{e}t}},G_{2}^{*}\e)\ra H^{\le
i}_{\rm{ab}}(S_{\e\rm{\acute{e}t}}, G_{1}^{*}\e)\ra H^{\le
i+1}_{\rm{ab}}(S_{\e\rm{\acute{e}t}},G_{3}^{*}\e)\ra\dots.
$$
\end{proposition}
\begin{proof} This follows from the above corollary and Theorem 3.14.
\end{proof}

\begin{remark} In \cite{bor2}, M. Borovoi has extended the constructions of Section 3
to an arbitrary non-empty base scheme and shown that the functor
$\pi_{1}(G)$ is defined and exact in this generality. Further, it
follows from [op.cit.] that $\pi_{1}\lbe( G
\e)\lbe\otimes^{\e\mathbf{L}} \bg_{m,S}$ and
$\rhom(\pi_{1}\lbe(G\e),\bz_{\! S})$ are quasi-isomorphic to
$\big(Z\big(\Gtil\,\big)\ra Z(G\e)\big)$ and $\big(Z(G\e)^{*}\ra
Z\big(\Gtil\,\big)^{*}\big)$, respectively. Consequently, both
Corollary 4.3 and Proposition 4.4 are valid over any non-empty base
scheme when $\pi_{1}(G)$ is defined as in \cite{bor2}.
\end{remark}

\begin{proposition} Let $G$ be a reductive group scheme over an
admissible scheme $S$, let $1\ra F \ra H \ra G \ra 1$ be a flasque
resolution of $G$ and set $R=H^{\lbe\tor}$. Assume that $R$ is split
by a connected Galois cover $S^{\e\prime}\ra S$ such that every
finite \'etale subcover $S^{\e\prime\prime}\ra S$ of
$S^{\e\prime}\ra S$ satisfies ${\rm{Pic}}\,S^{\e\prime\prime}=0$.
Then the given resolution defines an isomorphism
$$
H^{\le 1}_{\rm{ab}}(S_{\e\rm{fl}},G\e)\simeq\krn\!\!\left[ H^{\le
2}(S_{\e\rm{\acute{e}t}},F\e)\ra H^{\le
2}(S_{\e\rm{\acute{e}t}},R\e)\right].
$$
\end{proposition}
\begin{proof} Since $R$ is quasi-trivial, the hypothesis implies that
$H^{\le 1}(S_{\e\rm{\acute{e}t}},R\e)=0$ (cf. \cite{cts2}, proof of
Theorem 4.1, p.167). The result now follows from Proposition 4.2.
\end{proof}

\begin{remark}
By \cite{bou}, Proposition II.5.3.5, p.113 (see also
[op.cit.], Remark on p.112), the above proposition applies to any
reductive group scheme over the spectrum of an integral, noetherian and geometrically unibranched local
ring (e.g., a field).
\end{remark}

\appendix

\end{document}